\documentclass[12pt, english]{amsart}

\usepackage{geometry}
\usepackage[english]{babel}
\usepackage{multirow}
\usepackage{graphicx}
\usepackage{amssymb}
\usepackage{yhmath}
\usepackage{verbatim} 
\usepackage{amsmath}
\usepackage{amsthm}
\usepackage{color}
\usepackage{epstopdf}
\usepackage{epsfig}
\usepackage{wrapfig}
\usepackage[all]{xy}
\usepackage{hyperref}
\usepackage{enumitem} 

\newcommand{\vol}{{\rm vol}}

\newtheorem{theorem}{Theorem}
\newtheorem*{thmA}{Theorem A}
\newtheorem*{thmB}{Theorem B}

\newtheorem{coro}{Corollary}
\newtheorem{prop}{Proposition}
\newtheorem{lema}{Lemma}
\newtheorem*{defi}{Definition}

 \theoremstyle{remark}

\begin{document}
\title[Lower bound for the volume of unit vector fields on punctured spheres]{Poincar\'e index and the volume functional of unit vector fields on punctured spheres}
\author{Fabiano G. B. Brito}

\author{Andr\'e O. Gomes}

\author{Icaro Gon\c {c}alves}
\thanks{The third author is supported by a scholarship from the National Postdoctoral Program, PNPD-CAPES}

\address{Centro de Matem\'atica, Computa\c{c}\~ao e Cogni\c{c}\~ao,
Universidade Federal do ABC,
09.210-170 Santo Andr\'e, Brazil}
\email{fabiano.brito@ufabc.edu.br}

\address{Dpto. de Matem\'{a}tica, Instituto de Matem\'{a}tica e Estat\'{i}stica,
Universidade de S\={a}o Paulo, R. do Mat\={a}o 1010, S\={a}o Paulo-SP
05508-900, Brazil.}
\email{gomes@ime.usp.br}

\address{Dpto. de Matem\'{a}tica, Instituto de Matem\'{a}tica e Estat\'{i}stica,
Universidade de S\={a}o Paulo, R. do Mat\={a}o 1010, S\={a}o Paulo-SP
05508-900, Brazil.}
\email{icarog@ime.usp.br}


\begin{abstract}
For $n\geq 1$, we exhibit a lower bound for the volume of a unit vector field on $\mathbb{S}^{2n+1}\backslash\{\pm p\}$ depending on the absolute values of its Poincar\'e indices around $\pm p$. We determine which vector fields achieve this volume, and discuss the idea of having multiple isolated singularities of arbitrary configurations. 
\end{abstract}
\maketitle
\section{Introduction and statement of the results}
Let $M^m$ be a closed oriented Riemannian manifold and $\vec{v}$ a unit vector field on $M$. If $T^1M$ denotes the unit tangent bundle, endowed with the Sasaki metric, and regarding $\vec{v}: M \to T^1M$ as a smooth section,  the volume of $\vec{v}$ is defined as the volume of the submanifold $\vec{v}(M)\subset T^1M$, 
$$
\vol(\vec{v}) = \vol(\vec{v}(M)).
$$

On a given orthonormal local frame $\{e_1, \dots e_m\}$, there exists a formula (see \cite{GZ} and \cite{J}) in terms of the Riemannian metric of $M$. It reads
\begin{eqnarray}
\label{volume}
\vol(\vec{v}) &=& \int_M\sqrt{\det({\rm Id} + (\nabla \vec{v})^* (\nabla \vec{v}))}\nu\nonumber\\
&=& \int_M \Big(1+\sum_A \|\nabla_{e_A} \vec{v} \|^2  + \sum_{A<B} \| \nabla_{e_A} \vec{v}\wedge \nabla_{e_B} \vec{v} \|^2 + \cdots \nonumber \\
&\cdots& + \sum_{A_1<\cdots <A_{m-1}}\|\nabla_{e_{A_1}} \vec{v}\wedge \cdots \wedge \nabla_{e_{A_{m-1}}} \vec{v}  \|^2 
\Big)^{\frac{1}{2}}\nu, 
\end{eqnarray}
where $\nabla \vec{v}$ is an endomorphism of the tangent space at a given point, $\nu$ is the volume form of $M$ and $(\nabla \vec{v})^*$ is denotes adjoint operator.  
Intuitively speaking, the idea behind this functional is to measure which unit vectors are visually best organized, in the sense that those vectors would attain minimum possible value, \cite{GZ}. It is always true that $\vol(\vec{v})\geq \vol(M)$, and equality holds if and only if $\vec{v}$ is parallel with respect to $\nabla$. What makes worth looking for a minimum (or an infimum) for the volume is that not always a Riemannian manifold admits globally defined parallel vector fields, so in most cases the most symmetric organized unit vector field is not a trivial one, but rather a distinguished vector field. 

When Gluck and Ziller defined the volume functional, they proved that
\begin{theorem}[\cite{GZ}] The unit vector fields of minimum volume on $\mathbb{S}^3$ are precisely the Hopf vector fields, and no others.
\end{theorem}

Contrary to what the reader might expect, Hopf vector fields fail to minimize the volume functional in higher dimensional spheres,
\begin{theorem}[\cite{J}] Hopf fibrations on the round sphere $\mathbb{S}^5$ are not local minima of the volume functional. 
\end{theorem}

In pursuit of unit vector fields of minimum volume, several constructions stumbled on spheres minus one or minus a couple of points. One must keep in mind the following two examples, both of them defined on punctured spheres. 

The first example was given by Pedersen in \cite{P}, defined on a sphere minus one point. We denote it by $V_P$. It was shown in \cite{P} that its volume is 
$$
\vol(V_P) = \sqrt{2\pi n}\ \vol(\mathbb{S}^{2n+1}),
$$
for $n\geq 1$. The second example is a radial vector field on $\mathbb{S}^{2n+1}\backslash\{\pm p\}$. This vector field, denoted by $V_R$, is a geodesic vector field coming from the exponential map of the sphere at $p$. Brito {\it et al} proved the following
\begin{theorem}[\cite{BCN}] \label{BCN} Let $\vec{v}$ be a unit vector field on a compact Riemannian and oriented manifold $M^{2n+1}$. Then 
$$
\vol(\vec{v})\geq \int_M \left(\sum_{k=0}^{n} {n \choose k}{2n \choose 2k}^{-1} |\sigma_{2k}(\vec{v}^{\perp})| \right)\nu,
$$
where $\sigma_{2k}(\vec{v}^{\perp})$ is the $2k$-th elementary symmetric function of the second fundamental form of the distribution orthogonal to $\vec{v}$ (that is not necessarily integrable), with $\sigma_{0} = 1$. When $n\geq 2$, equality holds if and only if $\vec{v}$ is totally geodesic and $\vec{v}^{\perp}$ is integrable and umbilic. Furthermore, the following holds, 

(a) For every unit vector field $\vec{v}$ on $\mathbb{S}^{2n+1}$, 
$$
\vol(\vec{v}) \geq \sum_{k=0}^{n} {n \choose k}^2{2n \choose 2k}^{-1} \vol(\mathbb{S}^{2n+1}),
$$

and for $n\geq 2$ none of them achieves equality. 

(b) Let $\vec{v}$ be any non-singular unit vector field on $\vol(\mathbb{S}^{2n+1})$, then $\vol(V_R)\leq \vol(\vec{v})$. 
\end{theorem}

Besides, singular unit vector fields on $\mathbb{S}^2$ and the influence of the radius of a given sphere on the volume of Hopf vector fields have been studied, \cite{bo-gil} and \cite{bo-gil2}.

It can be shown that $\vol(V_R) = \frac{4^n}{{2n \choose n}}\vol(\mathbb{S}^{2n+1})$ (for example, see \cite{BCN}). Together with the value computed in \cite{GZ} for Hopf vector fields, $\vol(V_H) = 2^n\vol(\mathbb{S}^{2n+1})$, one is able to summarize some inequalities
$$
\vol(\mathbb{S}^{2n+1}) < \vol(V_R) < \vol(V_P) \ll \vol(V_H), 
$$
whenever $n\geq 2$. 

In addition, there are examples of how the topology of a vector field and the topology of the ambient space influence the volume. For Riemannian manifolds of dimension 5, Brito and Chac\'on \cite{BC} exhibited an inequality comparing the volume of a vector field to the Euler class of its orthonormal distribution. For Euclidean hypersurfaces, Reznikov \cite{R} deduced an inequality taking into account the degree of the Gauss map of the hypersurface. 

On the other hand, for antipodally punctured spheres of low dimensions, there is a relation regarding the index of the vector at the points $N = p$ and $S = -p$, 
\begin{theorem}[\cite{BCJ}] Let $\vec{v}$ be a unit smooth vector field defined on $\mathbb{S}^m\backslash\{N,S \}$. Then 

(a) for $m=2$, $\vol(\vec{v})\geq \frac{1}{2} (\pi + |I_{\vec{v}}(N)| + |I_{\vec{v}}(S)| -2)\vol(\mathbb{S}^2)$,

(b) for $m=3$, $\vol(\vec{v})\geq (|I_{\vec{v}}(N)| + |I_{\vec{v}}(S)|)\vol(\mathbb{S}^3)$,

\noindent
where $I_{\vec{v}}(P)$ stands for the Poincar\'e index of $\vec{v}$ around $P$. 
\end{theorem}

Our main goal is to extend the above result to higher odd dimensional spheres. The main theorem asserts 
\begin{thmA}
If $\vec{v}$ is a unit vector field on $\mathbb{S}^{2n+1}\backslash\{\pm p\}$, then
\begin{equation}
\vol(\vec{v})\geq \frac{\pi}{4}\vol(\mathbb{S}^{2n})\left(|I_{\vec{v}}(p)| + |I_{\vec{v}}(-p)| \right).
\end{equation}
\end{thmA}

In comparing the above estimate to the value achieved by radial vector fields, the following consequence is deduced

\begin{coro} For any unitary vector field $\vec{v}$ on $\mathbb{S}^{2n+1}\backslash\{\pm p\}$,
\begin{equation}
\vol(\vec{v})\geq \frac{\vol(V_R)}{2}\left(|I_{\vec{v}}(p)| + |I_{\vec{v}}(-p)| \right),\nonumber
\end{equation}
where $V_R$ denotes the north-south vector field. 
\end{coro}

The technique presented here can be exploited to obtain a straightforward extension to arbitrary isolated singularities, in a general Riemannian compact manifold 
\begin{thmB}
Let $\vec{v}$ be a unit vector field defined on $M^{2n+1}\backslash\{\cup_{i=1}^m p_i\}$, where $M$ is a compact Riemannian manifold and $\{p_i \}$ is a subset of isolated points. Then
\begin{equation}
\label{lower-arb}
\vol(\vec{v})\geq \frac{\vol(\mathbb{S}^{2n})}{2}\sum_{i=1}^{m}|I_{\vec{v}}(p_i)|
\end{equation}
\end{thmB}

This paper is organized as follows. We start Section 2 by introducing the Euler class of the normal bundle of $\vec{v}$, and then we define a list of functions depending on the vector field. We finish this Section by exhibiting an explicit representative of the Euler class. Section 3 is divided in five subsections, and in the last two of them we prove theorems A and B, respectively. Subsection \ref{poincare-index} is devoted to show how the indices of the vector field arise when the Euler class is restricted to small neighborhoods around its singularities. In Subsection \ref{inequalities} we briefly review some results from \cite{BCN} and use them to establish a comparison between the integrand in \ref{volume} and a function determined by the restriction of the Euler class. Last Section is dedicated to discuss the main theorems and future developments as well. 
\newpage
\section{Preliminaries and the Euler class}
Let $n\geq 1$ and set $M:=\mathbb{S}^{2n+1}\backslash\{\pm p \}$, endowed with the Riemannian metric $\langle \cdot, \cdot \rangle$. Let $\vec{v}$ be a unit vector field $\vec{v}: M\to T^1M$, and take $\{e_1, \dots, e_{2n}, e_{2n+1} = \vec{v} \}$ as an orthonormal local frame. We fix the following notation: $1\leq i,j,k,l,\dots \leq 2n$ and $1\leq A,B,C,D,\dots\leq 2n+1$. If $\{\omega_A \}$ is the associated local coframe, then the curvature and connection forms are related by the structure equations of $M$,
\begin{equation}
\omega_A(e_B) = \delta_{AB}, \quad \delta_{AB}=0\; {\rm if}\; A\neq B, \quad \delta_{AA}=1, \quad \nabla e_A = \sum_B \omega_{AB}e_B, \quad \omega_{AB}+\omega_{BA} = 0, \nonumber
\end{equation}
\begin{equation}
d\omega_A = \sum_B \omega_{AB}\wedge \omega_B, \quad d\omega_{AB} = \sum_C \omega_{AC}\wedge \omega_{CB} - \Omega_{AB}, \nonumber
\end{equation}
\begin{equation}
\Omega_{AB} = \frac{1}{2}\sum_{C,D}R_{ABCD}\omega_C \wedge \omega_D, \quad R_{ABCD}+ R_{ABDC}=0,\nonumber
\end{equation}

The normal bundle $\vec{v}^{\perp}$ is a subbundle of $TM$, and it admits a natural second fundamental form given locally by the matrix $(a_{ij})$, constructed with respect to the aforementioned local frame, $a_{AB} = \langle \nabla_{e_B} \vec{v}, e_A\rangle$. The curvature form of $\vec{v}^{\perp}$, $\Omega_{AB}^{\perp}$, is related to $\Omega_{AB}$ by means of the structure equations, 
\begin{equation}
\label{ee}
\Omega_{AB}^{\perp} = \Omega_{AB} + \omega_{A\, 2n+1} \wedge \omega_{B\, 2n+1}.
\end{equation}

We recall the definition of the Euler form in terms of the Pfaffian of $\Omega_{AB}^{\perp}$, 
\begin{equation}
\mathcal{E}(\vec{v}^{\perp}) = \frac{2}{(2n)!\vol(\mathbb{S}^{2n})}\sum_{\sigma\in \mathcal{S}_{2n}}{\rm sgn}(\sigma) \Omega_{\sigma(1)\sigma(2)}^{\perp}\wedge\cdots \wedge \Omega_{\sigma(2n-1)\sigma(2n)}^{\perp},
\end{equation}
where $\mathcal{S}_{2n}$ stands for the permutation group of $2n$ elements while ${\rm sgn}(\sigma)$ equals the sign of $\sigma$. 

Before computing $\mathcal{E}(\vec{v}^{\perp})$ we need to settle our notation. For each $1\leq i\leq 2n$, we say that $\sigma_i$ is the $i$-th elementary symmetric function of the matrix $(a_{ij})$. The function $\sigma_i$ is the sum of all $i\times i$ minors from $(a_{ij})$. 

The last column of $(a_{AB})$ has some special meaning. It is formed by the elements $a_{i\,2n+1} = \langle \nabla_{\vec{v}} \vec{v}, e_i\rangle$, which are components of the acceleration of $\vec{v}$. We employ these components in the next definition.

\begin{defi}
\label{sym}
Let $(a_{ij}(l))$ denote the $2n\times 2n$ matrix obtained from $(a_{ij})$ by changing its $l$-th column with the components of $\nabla_{\vec{v}} \vec{v}$, 
\[
(a_{ij}(l)) = 
  \left(\begin{array}{cccccccc}
      a_{11} & \cdots & a_{1\, l-1}  & a_{1\, 2n+1}  & a_{1\, l+1}  &\cdots& a_{1\,2n}\\
      \vdots &        & \vdots       & \vdots        & \vdots       && \vdots \\
  a_{2n\, 1} & \cdots & a_{2n\, l-1} & a_{2n\, 2n+1} & a_{2n\, l+1} &\cdots& a_{2n\,2n}  
  \end{array}\right).
\]

We say that $\sigma_i^{\perp}(l)$ is the sum of all $i \times i$ minors of the matrix $(a_{ij}(l))$ having at least one element depending on $\nabla_{\vec{v}} \vec{v}$. 
\end{defi}
For example, $\sigma_2^{\perp}(2n)$ is the sum of all $2\times 2$ minors of $a_{ij}(2n)$ such that at least one of their columns is made of components of $\nabla_{\vec{v}} \vec{v}$,

$$ \sigma_2^{\perp}(2n) = \sum_{\substack{j=1\\1\leq i<k\leq 2n-1}}^{2n}
\det\begin{bmatrix}
a_{ij} & a_{i\,2n+1} \\ 
a_{kj}  & a_{k\,2n+1}
\end{bmatrix}.
$$
It is important that we distinguish the functions $\sigma_i^{\perp}(l)$ from the symmetric elementary functions of $(a_{ij}(l))$, say $\sigma_i(l)$. The former is just a part of the latter, and they naturally appear when computing the Euler class of $\vec{v}^{\perp}$.

\begin{lema} The Euler class $\mathcal{E}(\vec{v}^{\perp})\in H^{2n}(M, \mathbb{R})=H^{2n}(\mathbb{S}^{2n+1}\backslash\{\pm p \}, \mathbb{R}) \cong \mathbb{R}$ can be represented by the following element
\label{lema-euler}
\begin{equation}
\label{euler}
\mathcal{E}(\vec{v}^{\perp}) = \frac{2}{\vol(\mathbb{S}^{2n})} \sum_{k=0}^{n}{n\choose k}{2n\choose 2k}^{-1}W(k),
\end{equation}
where, denoting $\widehat{\omega}$ the omitted term,  
\begin{eqnarray}
W(k)\! &=& \sum_C \sigma_{2k}^{\perp}(C)\omega_1\wedge\cdots\wedge\widehat{\omega_C}\wedge\cdots\wedge\omega_{2n+1} \nonumber\\ &=&  \sum_l \sigma_{2k}^{\perp}(l)\omega_1\wedge\cdots\wedge\widehat{\omega_l}\wedge\cdots\wedge\omega_{2n+1} + \sigma_{2k}\omega_1 \wedge\cdots\wedge\omega_{2n}. 
\nonumber
\end{eqnarray}
\end{lema}
\begin{proof}

The fact that $\Omega_{AB} = \omega_A\wedge\omega_B$ (the metric on $M$ is just the restriction of the round Riemannian metric of $\mathbb{S}^{2n+1}$) together with a nice rearrangement of terms imply  
\begin{eqnarray}
\mathcal{E}(\vec{v}^{\perp})\!\! &=&\!\! \frac{2}{(2n)!\vol(\mathbb{S}^{2n})}\sum_{\sigma\in \mathcal{S}_{2n}}{\rm sgn}(\sigma)\sum_{k=0}^{n}{n\choose k}
\omega_{\sigma(1)}\wedge\cdots\wedge \omega_{\sigma(2k)}\wedge\omega_{\sigma(2k+1)\, 2n+1}\wedge\cdots\wedge \omega_{\sigma(2n)\, 2n+1}.\nonumber
\end{eqnarray}

Taking the second fundamental form of $\vec{v}^{\perp}$ into account, we write $\omega_{A\, 2n+1} = -\sum_Ba_{AB}\omega_B$, and consequently  $\omega_{A\, 2n+1}\wedge \omega_{B\, 2n+1} = 
\sum_{C,D} a_{AC} a_{BD}\omega_{C} \wedge\omega_{D}
$. Hence
\begin{eqnarray}
\mathcal{E}(\vec{v}^{\perp}) &=& \frac{2}{(2n)!\vol(\mathbb{S}^{2n})}\sum_{\sigma\in \mathcal{S}_{2n}}{\rm sgn}(\sigma)\sum_{k=0}^{n}{n\choose k}
\omega_{\sigma(1)}\wedge\cdots\wedge \omega_{\sigma(2k)}\nonumber \\
&\wedge& \left(\sum_{B_1} a_{\sigma(2k+1)B_1}\omega_{B_1}\right)\wedge\cdots\wedge \left(\sum_{B_{2(n-k)}} a_{\sigma(2n)B_{2(n-k)}}\omega_{B_{2(n-k)}}\right).\nonumber
\end{eqnarray}

Now it is a matter of separating the coefficients of $2n$-forms $\omega_{A_1}\wedge \cdots\wedge\omega_{A_{2n}}$.

When we fix those $2n$-forms, we have to count them within all permutations in $\mathcal{S}_{2n}$. For example, $k=1$ gives us the following summand 
$$
\sum_{\sigma\in \mathcal{S}_{2n}}{\rm sgn}(\sigma)\omega_{\sigma(1)}\wedge\omega_{\sigma(2)}
\wedge\left(\sum_{B_1} a_{\sigma(3)B_1}\omega_{B_1}\right)\wedge\cdots\wedge \left(\sum_{B_{2(n-1)}} a_{\sigma(2n)B_{2(n-1)}}\omega_{B_{2(n-1)}}\right).
$$ 

Consequently, we end up with a number, $(2n-2k)!(2k)!$, and since the Pfaffian is divided by $(2n)!$ we have that $\frac{(2n-2k)!(2k)!}{(2n)!} = {2n\choose 2k}^{-1}$.   

On the other hand, the products $a_{\sigma(2k+1)B_1}\cdots a_{\sigma(2n)B_{2(n-k)}}$ from
$$
\left(\sum_{B_1} a_{\sigma(2k+1)B_1}\omega_{B_1}\right)\wedge\cdots\wedge \left(\sum_{B_{2(n-k)}} a_{\sigma(2n)B_{2(n-k)}}\omega_{B_{2(n-k)}}\right)
$$
determine some minors coming from the matrix $(a_{AB})$. Functions like $\sigma_i^{\perp}(\cdot)$ from definition \ref{sym} appear every time $B_{i} = 2n+1$, for some $i$, and this happens in all terms except in the coefficient of $\omega_1\wedge\cdots\wedge\omega_{2n}$, which is accompanied by the elementary symmetric functions of $(a_{ij})$. Finally, it is a matter of separating those minors according to the $2n$-form which multiplies them. 

\end{proof}

\section{Development towards demonstrating theorems A and B}
\subsection{Poincar\'e index}
\label{poincare-index}
Let $\mathbb{S}^{2n}_{\theta}$ be a parallel of latitude $\theta \in (-\frac{\pi}{2}, \frac{\pi}{2})$ and let $\iota = \iota_{\theta}:\mathbb{S}^{2n}_{\theta}\to M $ be its natural embedding. We may assume that $p$ belongs to the northern hemisphere of $\mathbb{S}^{2n+1}$, while $-p$ is in the southern hemisphere. Given $\epsilon > 0$, $\mathbb{S}^{2n}_{\frac{\pi}{2} - \epsilon}$ is a small parallel near $p$, and together with $\mathbb{S}^{2n}_{\theta}$ we have an associated annulus region $A_{\theta,\,\epsilon}^{2n}$ of dimension $2n$, with boundary $\mathbb{S}^{2n}_{\frac{\pi}{2} - \epsilon}\cup \mathbb{S}^{2n}_{\theta}$; see figure \ref{fig:sphere-sn}. 
\begin{figure}
\centering
\includegraphics[width=0.4\linewidth]{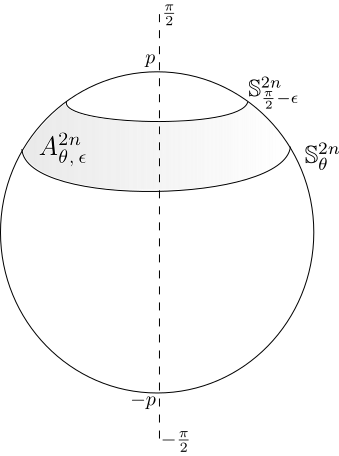}
\caption{$\mathbb{S}^{2n+1}$ with an annulus region near north pole.}
\label{fig:sphere-sn}
\end{figure}

By Stokes' theorem, 
$$
\int_{A_{\theta,\,\epsilon}^{2n}} d\,\iota^*(\mathcal{E}(\vec{v}^{\perp})) = \int_{\mathbb{S}^{2n}_{\frac{\pi}{2} - \epsilon}\cup\ \mathbb{S}^{2n}_{\theta}} \iota^*(\mathcal{E}(\vec{v}^{\perp})).
$$
However, $\mathcal{E}(\vec{v}^{\perp})$ is closed, so $d\,\iota^*(\mathcal{E}(\vec{v}^{\perp})) = 0$ and we conclude that the integrals of its restrictions to both spheres are equal, 
\begin{equation}
\label{st-euler}
\int_{\mathbb{S}^{2n}_{\frac{\pi}{2} - \epsilon}} \iota^*(\mathcal{E}(\vec{v}^{\perp})) = \int_{ \mathbb{S}^{2n}_{\theta}} \iota^*(\mathcal{E}(\vec{v}^{\perp})).
\end{equation}

Next we compute the restriction of $\iota^*(\mathcal{E}(\vec{v}^{\perp}))$ on $\mathbb{S}^{2n}_{\theta}$. 

We may suppose that $e_1, \dots,e_{2n-1}$ are all tangent to $\mathbb{S}^{2n}_{\theta}$. Let $\alpha \in [0,2\pi]$ be the oriented angle from the tangent space of $\mathbb{S}^{2n}_{\theta}$ to $\vec{v}$. In this case, $\{e_1, \dots,e_{2n-1},u := \sin\alpha e_{2n} + \cos \alpha \vec{v} \}$ is an orthonormal positively oriented local frame on $\mathbb{S}^{2n}_{\theta}$. 

Fix $0\leq k \leq n$. Following equation \ref{euler} of Lemma \ref{lema-euler}, we decompose $W(k)$ as follows
\begin{eqnarray}
W(k) &=& \sum_{l=1}^{2n-1} \sigma_{2k}^{\perp}(l)\omega_1\wedge\cdots\wedge\widehat{\omega_l}\wedge\cdots\wedge\omega_{2n+1} \nonumber\\ &+& \sigma_{2k}^{\perp}(2n)\omega_1 \wedge\cdots\wedge\omega_{2n-1}\wedge \omega_{2n+1} +\sigma_{2k}\omega_1 \wedge\cdots\wedge\omega_{2n}. 
\nonumber
\end{eqnarray}
By applying $W(k)$ on $(e_1, \dots,e_{2n-1},u)$, we see that $$\omega_1\wedge\cdots\wedge\widehat{\omega_l}\wedge\cdots\wedge\omega_{2n+1} (e_1, \dots,e_{2n-1},u) = 0,$$ when $1\leq l \leq 2n-1$, because $e_l$ is in $(e_1, \dots,e_{2n-1},u)$ but $\omega_l$ is omitted. Thus, just the last two terms remain, i.e., $$W(k)(e_1, \dots,e_{2n-1},u) = \sin\alpha\,\sigma_{2k} + \cos\alpha\, \sigma^{\perp}_{2k}(2n).$$ Therefore, 
\begin{equation}
\label{pullback-euler}
\iota^*(\mathcal{E}(\vec{v}^{\perp})) = \frac{2}{\vol(\mathbb{S}^{2n})} \sum_{k=0}^{n}{n\choose k}{2n\choose 2k}^{-1}\left(\sin\alpha\,\sigma_{2k} + \cos\alpha\, \sigma^{\perp}_{2k}(2n)\right) \, \nu_{\mathbb{S}^{2n}_{\theta}}.
\end{equation}

Going back to \ref{st-euler}, its right hand side remains unchanged when we take the limit as $\epsilon$ goes to zero. Nevertheless, its left hand side is an integral of a function similar to the one appearing in \ref{pullback-euler}, but for a different angle, since this angle depends on latitude of the parallel $\mathbb{S}^{2n}_{\frac{\pi}{2}-\epsilon}$, and of course on the vector $\vec{v}$. Thus, as $\epsilon$ goes to zero the only non-vanishing term comes from the restriction of $\vec{v}$ to $\mathbb{S}^{2n}_{\frac{\pi}{2} - \epsilon}$, which is the degree of $\vec{v}: \mathbb{S}^{2n}_{\frac{\pi}{2} - \epsilon} \to \mathbb{S}^{2n}$, and this degree equals the Poncar\'e index around $p$ (cf. \cite{C}).  Therefore, 
\begin{equation}
\label{index-N}
\lim_{\epsilon\to 0} \int_{\mathbb{S}^{2n}_{\frac{\pi}{2} - \epsilon}} \iota^*(\mathcal{E}(\vec{v}^{\perp})) = I_{\vec{v}}(p). \end{equation}
Following a similar argument, 
\begin{equation}
\label{index-S}\lim_{\epsilon\to 0} \int_{\mathbb{S}^{2n}_{-\frac{\pi}{2} + \epsilon}} \iota^*(\mathcal{E}(\vec{v}^{\perp})) = I_{\vec{v}}(-p). 
\end{equation}

\subsection{Inequalities: volume of a matrix}
\label{inequalities}
Our previous discussion determines how the Euler form relates to the volume form of $\mathbb{S}^{2n}_{\theta}$, and when the Poincar\'e indices of $\vec{v}$ arise when a representative of the Euler class of $\vec{v}^{\perp}$ restricts to small neighborhoods around $\pm p$. Now we compare the function on \ref{pullback-euler} to $\sqrt{\det({\rm Id} + (\nabla \vec{v})^* (\nabla \vec{v}))}$. 

Following \cite{BCN}, the volume of a linear transformation $T:V^m\to V^m$ is the volume of the graph of the cube under $T$. Equivalently,

\begin{prop}[\cite{BCN}] Let $T$ be an endomorphism and $B=(b_{ij})$ the matrix of $T$ associated to some orthonormal basis. Then
\begin{equation}
\label{volT}
\vol(T) = \Bigg(1+\sum_{1\leq i,j\leq m}b_{ij}^2 + \sum_{\substack{i_1<i_2\\ j_1<j_2}}\left(\det B^{i_1i_2}_{j_1j_2}\right)^2 + \cdots + \sum_{\substack{i_1<\cdots < i_{m-1}\\ j_1<\cdots < j_{m-1}}}\left(\det B^{i_1\cdots  i_{m-1}}_{j_1\cdots  j_{m-1}}\right)^2 + (\det B)^2                          \Bigg)^{\frac{1}{2}},   \nonumber
\end{equation}
where $B^{i_1\cdots  i_{k}}_{j_1\cdots  j_{k}}$ is the submatrix of $B$ corresponding to the rows $i_1\cdots  i_{k}$ and columns $j_1\cdots  j_{k}$.
\end{prop}

In order to prove theorem \ref{BCN}, the authors compared the volume of a given $2m\times 2m$ diagonal matrix $D$ (with nonnegative entries) to the sum of its elementary symmetric functions. They proved an algebraic inequality (it comes from ``Fundamental Lemma", Section 3 of \cite{BCN})
\begin{equation}
\label{vol-comp}
\vol(D)\geq \left(\sum_{k=0}^{m} {m \choose k}{2m \choose 2k}^{-1} \sigma_{2k}(D) \right). 
\end{equation}

Our goal is to exhibit a matrix of even dimension such that its volume coincide with $\sqrt{\det({\rm Id} + (\nabla \vec{v})^* (\nabla \vec{v}))}$ and its elementary symmetric functions are directly related (or can be compared) to the sum $\sigma_{2k} + \sigma^{\perp}_{2k}(2n)$.

When we fix an orthonormal local frame $\{e_1, \dots, e_{2n}, \vec{v} \}$, we have an associated $(2n+1)\times (2n+1)$ matrix $(a_{AB}) = (\langle \nabla_{e_B} \vec{v}, e_A\rangle)$,
\[
(a_{AB}) = 
  \left(\begin{array}{ccc|c}
    \multicolumn{3}{c|}{\multirow{3}{*}{\raisebox{2pt}{$(a_{ij})$}}}     & a_{1\, 2n+1}           \\ 
    &       &            & {\vdots}    \\
    &       &            & a_{2n\, 2n+1}          \\ \hline
    0      & \cdots & 0  & 0
  \end{array}\right).
\]
Notice that the last row is zero since $\vec{v}$ is a unitary vector field. 

\begin{lema} According to the notation settled above,
\begin{equation}
\label{sigmas-volume}
\sqrt{\det({\rm Id} + (\nabla \vec{v})^* (\nabla \vec{v}))}\geq \sum_{k=0}^{n} {n \choose k}{2n \choose 2k}^{-1} \left(|\sigma_{2k}| +  |\sigma^{\perp}_{2k}(2n)|\right).
\end{equation}
\end{lema}
\begin{proof}
We define a $(2n+2)\times (2n+2)$ matrix $(b_{AB})$ by adding to $(a_{AB})$ a column and a row of zeros, 
\[
(b_{AB}) = 
  \left(\begin{array}{cccc|c}
    \multicolumn{3}{c}{\multirow{3}{*}{\raisebox{2pt}{$(a_{ij})$}}}     & a_{1\, 2n+1}  & 0         \\ 
    &       &            & {\vdots} &  {\vdots}  \\
    &       &            & a_{2n\, 2n+1} & 0          \\ 
    0      & \cdots & 0  & 0 & 0\\
    \hline
    0      & \cdots & 0  & 0 & 0
  \end{array}\right),
\]
so
$$\vol(b_{AB}) = \vol(a_{AB}) = \sqrt{\det({\rm Id} + (\nabla \vec{v})^* (\nabla \vec{v}))}.$$

By changing the basis, we can write $(b_{AB})$ as a upper triangular matrix, having its eigenvalues in the main diagonal (some of them possibly complex) 
\[(b_{AB}) = 
    \left(
    \begin{array}{cccccccccc}
    \lambda_1  &\ast & \cdots &&&&&&\cdots& \ast \\
    0  & \ddots    &   & &&&&&&\vdots\\
    \vdots  &  &\ \lambda_r & \ast & \cdots  \\
      &  & 0   &             \\
      & & \vdots & & x_1 & -y_1 &\ast & \cdots \\
      & & & & y_1&\ x_1\\
      & & & & 0& & \ddots &&&\vdots\\
      & & & & \vdots & & & \ddots & & \ast \\
     \vdots & & & & & & & & x_s & -y_s\\
     0 &\cdots & & & & & \cdots& 0&y_s &\ x_s
    \end{array}
    \right).
\]  

In general, $(a_{ij})$ is not a symmetric matrix, since $\vec{v}^{\perp}$ is not necessarily integrable. Thus, even though $(b_{AB})$ is possible a non-diagonal matrix, it has at least two zero eigenvalues, say $\lambda_1$ and $\lambda_2$,
and this fact plays a role when counting its elementary symmetric functions. 
If we define $D = \text{diagonal}(0,0,|\lambda_3|,\dots, \dots,|\lambda_r|,\sqrt{x_1^2 + y_1^2}, \sqrt{x_1^2 + y_1^2}, \dots, \sqrt{x_s^2 + y_s^2}, \sqrt{x_s^2 + y_s^2})$, then \ref{vol-comp} holds for this diagonal matrix. Summation goes up to $n$ instead of $n+1$ simply because $D$ is equivalent to a $2n\times 2n$ matrix. The fact that $(b_{AB})$ has elements above its main diagonal implies that $\vol(b_{AB}) \geq \vol(D)$. Since $D$ has nonnegative entries, $\sigma_{2k}(D) \geq \sigma_{2k}((b_{AB}))$ (cf. \cite{BCN}, Sections 3 and 4). Therefore omitting the symmetric functions $\sigma_{2k}^{\perp}(l)$, for $1\leq l\leq 2n-1$ produces the desired inequality
\begin{equation}
\vol(b_{AB})\geq \sum_{k=0}^{n} {n \choose k}{2n \choose 2k}^{-1} \sigma_{2k}(b_{AB}) \geq \sum_{k=0}^{n} {n \choose k}{2n \choose 2k}^{-1} \left(\sigma_{2k} +  \sigma^{\perp}_{2k}(2n)\right). \nonumber
\end{equation}

\end{proof}

\subsection{Proof of theorem A}
\label{proofs}
We split the integral \ref{volume} on $M$ as an integral on a parallel $\mathbb{S}^{2n}_{\theta}$ of latitude $\theta \in (-\frac{\pi}{2}, \frac{\pi}{2})$, and a integral on $\theta$ itself,
$$\vol(\vec{v}) = \int_M \sqrt{\det({\rm Id} + (\nabla \vec{v})^* (\nabla \vec{v}))} \nu_M = \int_{-\frac{\pi}{2}}^{\frac{\pi}{2}}\left(\int_{\mathbb{S}^{2n}_{\theta}}\sqrt{\det({\rm Id} + (\nabla \vec{v})^* (\nabla \vec{v}))}\nu_{\mathbb{S}^{2n}_{\theta}}\right)d\theta.$$
From equation \ref{sigmas-volume}, 
$$\vol(\vec{v}) \geq \sum_{k=0}^{n} {n \choose k}{2n \choose 2k}^{-1}  \int_{-\frac{\pi}{2}}^{\frac{\pi}{2}} \left(\int_{\mathbb{S}^{2n}_{\theta}}\left(|\sigma_{2k}| +  |\sigma^{\perp}_{2k}(2n)|\right)\nu_{\mathbb{S}^{2n}_{\theta}}\right)d\theta$$
Since $\sin$ and $\cos$ are bounded,  
\begin{eqnarray}
\sum_{k=0}^{n}\frac{{n\choose k}}{{2n\choose 2k}}\left(\sin\alpha\,\sigma_{2k} + \cos\alpha\, \sigma^{\perp}_{2k}(2n)\right)
&\leq&\sum_{k=0}^{n}\frac{{n\choose k}}{{2n\choose 2k}}|\sigma_{2k}| + \sum_{k=0}^{n}\frac{{n\choose k}}{{2n\choose 2k}}\left|\sigma^{\perp}_{2k}(2n)\right|,\nonumber
\end{eqnarray}
and then, from equations \ref{pullback-euler} and \ref{st-euler},
$$\vol(\vec{v}) \geq \frac{\vol(\mathbb{S}^{2n})}{2} \int_{-\frac{\pi}{2}}^{\frac{\pi}{2}}\int_{ \mathbb{S}^{2n}_{\theta}} \iota^*(\mathcal{E}(\vec{v}^{\perp})) = \frac{\vol(\mathbb{S}^{2n})}{2}\left(  \int_{-\frac{\pi}{2}}^{0}\int_{\mathbb{S}^{2n}_{-\frac{\pi}{2} +\epsilon}} \iota^*(\mathcal{E}(\vec{v}^{\perp})) + \int_{0}^{\frac{\pi}{2}}\int_{\mathbb{S}^{2n}_{\frac{\pi}{2} - \epsilon}} \iota^*(\mathcal{E}(\vec{v}^{\perp}))\right)
$$
Therefore, 
$$\vol(\vec{v}) \geq\frac{\pi}{4}\vol(\mathbb{S}^{2n})\left(|I_{\vec{v}}(p)| + |I_{\vec{v}}(-p)| \right),$$
which proves theorem A. 

\subsection{A modest extension to arbitrary isolated singularities: proof of theorem B}
For every $p_i$, $1\leq i\leq m$, we can take the exponential map on $T_{p_i}M$ and find a real number $\theta_i$ such that a geodesic sphere  $S^{2n}_{\theta_i}$ is the boundary of a geodesic ball in $M^{2n+1}$, centered in $p_i$ and containing one singularity, namely $p_i$. 

Given $\epsilon_i > 0$ smaller than $\theta_i$, we build an annulus region $A_{\theta_i,\,\epsilon_i}^{2n}$ of dimension $2n$, with boundary $S^{2n}_{\epsilon_i}\cup S^{2n}_{\theta_i}$. Figure \ref{fig:variosb} illustrates the idea when we restrict ourselves to the case $M = \mathbb{S}^{2n+1}$. We proceed as in subsection \ref{poincare-index}. 

We merely consider that 
$$
\int_{M}\sqrt{\det({\rm Id} + (\nabla \vec{v})^* (\nabla \vec{v}))}\geq \sum_{i} \int_{S^{2n}_{\theta_i}} \sqrt{\det({\rm Id} + (\nabla \vec{v})^* (\nabla \vec{v}))}
$$
In this case, inequality \ref{sigmas-volume} still holds. Therefore,
$$\vol(\vec{v}) \geq \frac{\vol(\mathbb{S}^{2n})}{2} \sum_{i=1}^{m} \int_{ S^{2n}_{\theta_i}} \iota^*(\mathcal{E}(\vec{v}^{\perp})) = \frac{\vol(\mathbb{S}^{2n})}{2}\sum_{i=1}^{m}|I_{\vec{v}}(p_i)|$$

\begin{figure}
\centering
\includegraphics[width=0.4\linewidth]{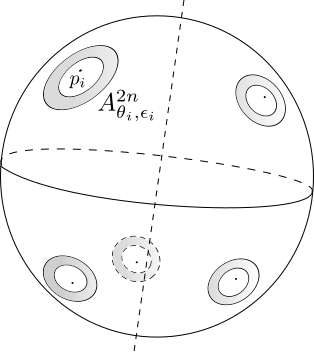}
\caption{A sphere with various isolated points, each one having a small annulus region around it.}
\label{fig:variosb}
\end{figure}

\section{Concluding remarks}
Even though, compared to theorem A, the lower bound found in \ref{lower-arb} is not sharp when $m=2$ and $M = \mathbb{S}^{2n+1}\backslash \{\pm p \}$, it presents a lower value for vector fields having two singularities in a random position, rather than on antipodal points. 

Additionally, as discussed in \cite{chacon-nunes} for the energy functional, given a number (greater than two) of isolated singularities, it is possible to find a unit vector field having these singularities and with volume arbitrarily close to the volume of the radial vector field. This may be done by the following argument: put two singularities in antipodal points $\pm p$ and every remain singularity in a neighborhood near the south pole $-p$, for example. Outside this neighborhood, take the radial vector field coming from $p$ and inside it one can take any vector field preserving the indices that were established in the beginning. By gluing those two parts together, one can obtain a vector field such that its volume is close to the volume of $V_R$. This is possible since the smaller the neighborhood, the smaller the volume.  

Theorem B represents a fair topological step towards a more general geometric question: is it possible to determined a unit vector field of minimum volume on a Riemannian manifold without a subset of singularities in a fixed configuration?

\end{document}